\documentclass[12pt]{amsart}

\newtheorem{theorem}{Theorem}
\newtheorem{corollary}[theorem]{Corollary}

\newtheorem{question}[theorem]{Question}
\newtheorem{lemma}[theorem]{Lemma}

\begin{document}

\title{Free sequences and the tightness of pseudoradial spaces}

\author[S. Spadaro]{Santi Spadaro}
\address{Department of Mathematics and Computer Science\\
University of Catania\\
viale A. Doria 6\\
95125 Catania, Itay}
\email{santidspadaro@gmail.com}


\maketitle

\begin{abstract}
Let $F(X)$ be the supremum of cardinalities of free sequences in $X$. We prove that the radial character of every Lindel\"of Hausdorff almost radial space $X$ and the set-tightness of every Lindel\"of Hausdorff space are always bounded above by $F(X)$. We then improve a result of Dow, Juh\'asz, Soukup, Szentmikl\'ossy and Weiss by proving that if $X$ is a Lindel\"of Hausdorff space, and $X_\delta$ denotes the $G_\delta$ topology on $X$ then $t(X_\delta) \leq 2^{t(X)}$. Finally, we exploit this to prove that if $X$ is a Lindel\"of Hausdorff pseudoradial space then $F(X_\delta) \leq 2^{F(X)}$.
\end{abstract}

\maketitle 

\section{Introduction}

Free sequences were one of the key tools in Arhangel'skii's celebrated solution of the Alexandroff-Urysohn problem on the cardinality of first-countable compacta. Later, they were discovered to have a definitive impact in various other aspects of the theory of cardinal invariants in topology. Recall that a sequence $\{x_\alpha: \alpha < \kappa \}$ in a topological space $X$ is said to be \emph{free} if $\overline{\{x_\alpha: \alpha < \beta \}} \cap \overline{\{x_\alpha: \beta \leq \alpha < \kappa \}}=\emptyset$, for every $\beta < \kappa$. 

Let $F(X)$ be the supremum of the cardinalities of free sequences in $X$ and let $t(X)$ be the tightness of $X$. Arhangel'skii proved that $F(X)=t(X)$ for every compact Hausdorff space $X$. For Lindel\"of spaces this is not true anymore, as first noted by Okunev \cite{Ok}. Indeed, let $\sigma(2^\kappa)=\{x \in 2^{\kappa}: |x^{-1}(1)| < \aleph_0 \}$ and let $\mathbf{1} \in 2^\kappa$ be the function that is constantly equal to $1$. Then $X=\sigma(2^\kappa) \cup \{\mathbf{1}\}$ with the topology induced from $2^\kappa$ is a space of tightness $\kappa$ where free sequences are countable (the last assertion follows easily from the fact that $\sigma(2^\kappa)$ is a $\sigma$-compact space of countable tightness). However, the inequality $F(X) \leq t(X)$ is still true for every Lindel\"of space $X$. More generally, if $L(X)$ denotes the \emph{Lindel\"of degree} of $X$, then $F(X) \leq L(X) \cdot t(X)$.

There is at least another class of spaces for which we can prove the equality $F(X)=t(X)$. A topological space $X$ is called \emph{pseudoradial} if for every non-closed set $A \subset X$ there is a point $x \in \overline{A} \setminus A$ and a transfinite sequence $\{x_\alpha: \alpha < \kappa \} \subset A$ which converges to $x$. If for every non-closed set $A \subset X$ and every point $x \in \overline{A} \setminus A$ there is a transfinite sequence inside $A$ which converges to $x$, then the space $X$ is called \emph{radial}. In \cite{B}, Bella proved that $t(X) \leq F(X)$ for every pseudoradial regular space $X$ and hence $t(X)=F(X)$ for every Lindel\"of pseudoradial regular space $X$. 

Given a pseudoradial space $X$ the \emph{radial character of $X$}, $\sigma_C(X)$ is defined as the minimum cardinal $\kappa$ such that every non-closed set contains a transfinite sequence of length at most $\kappa$ converging outside. For radial spaces obviously $\sigma_C(X)=t(X)$, but this is true for an even larger class of spaces, that of almost radial spaces. A space $X$ is said to be \emph{almost radial} if, for every non-closed set $A \subset X$ there is a point $x \in \overline{A} \setminus A$ and a sequence $\{x_\alpha: \alpha < \kappa \}$ converging to $x$ such that $x \notin \overline{\{x_\alpha: \alpha < \beta\}}$, for every $\beta < \kappa$ (we will call such a sequence a \emph{thin sequence}). Bella's result implies that $t(X)=\sigma_C(X) \leq F(X)$ for every regular almost radial space $X$.

Since regularity was essential in the proof of his results, Bella \cite{B} asked the following question:

\begin{question}
Let $X$ be a Hausdorff pseudoradial space. Is it true that $t(X) \leq F(X)$? Is it at least true that $\sigma_C(X) \leq F(X)$ for every Hausdorff almost radial space $X$?
\end{question}

The above question is still open. We will prove that $t(X) \leq L(X) \cdot F(X)$ for every Hausdorff pseudoradial space and this will allow us to prove that $t(X)=F(X)$ for every Lindel\"of pseudoradial Hausdorff space, thus removing the regularity assumption from one of Bella's results. Using the notion of \emph{set-tightness} we will be able to extend our results outside of the pseudoradial realm.

In the second part of the paper we will prove a bound for the  tightness of the $G_\delta$-modification of a Lindel\"of Hausdorff space. The $G_\delta$-modification of $X$ (or $G_\delta$ topology) is the topology on $X$ generated by the $G_\delta$-subsets of $X$. Studying which properties of $X$ are preserved by this operation is a natural question that has been extensively studied in the literature. Of course the $G_\delta$-topology is compact if and only if the space is finite, but there are already some interesting preservation results for the Lindel\"of property. For example, Arhangel'skii proved that if $X$ is a Lindel\"of scattered space then $X_\delta$ is Lindel\"of. More generally, one can ask whether the cardinal functions of $X_\delta$ can be bounded in terms of their value on $X$. Sometimes a neat ZFC bound is available. For example, Juh\'asz \cite{JArhan} proved that $c(X_\delta) \leq 2^{c(X)}$ for every compact Hausdorff space $X$, where $c(X)$ denotes the \emph{cellularity} of $X$, Bella and the author \cite{BS} proved that $s(X_\delta) \leq 2^{s(X)}$ for every Hausdorff space $X$, where $s(X)$ is the \emph{spread} of $X$ and Carlson, Porter and Ridderbos \cite{CPR} showed that $L(X_\delta) \leq 2^{F(X) \cdot L(X)}$ for every Hausdorff space $X$ (the $F(X)$ in the exponent cannot be removed in the last result as there is even a compact space $X$ such that $wL(X_\delta) > 2^{\aleph_0}$, see \cite{SSz}). In other cases the existence of a bound may strongly depend on your set theory. For example, the authors of \cite{KMS} proved that if $Nt(X)$ denotes the \emph{Noetherian type} of $X$ then $Nt(X_\delta) \leq 2^{Nt(X)}$ for every compact $X$ if GCH is assumed and $Nt(X)$ has uncountable cofinality and showed, modulo very large cardinals, that both of their assumptions where essential.

The behavior of the tightness of the $G_\delta$ topology changes drastically depending on whether the space is Lindel\"of or not. Dow, Juh\'asz, Soukup, Szentmikl\'ossy and Weiss proved that, consistently, there can be countably tight regular spaces whose $G_\delta$ topology has arbitrarily large tightness, but for every Lindel\"of regular space $X$ we have $t(X_\delta) \leq 2^{t(X)}$. 

Regularity was essential in their argument, but we will show how to remove it, and as a byproduct we will obtain that the natural bound $F(X_\delta) \leq 2^{F(X)}$ holds for a Lindel\"of pseudoradial Hausdorff space $X$.

For undefined notions we refer to \cite{E} and \cite{J}. Background material on elementary submodels, which are used in the proof of Theorem $\ref{tightthm}$, can be found in Dow's survey \cite{DElem}. For general information about pseudoradial spaces we recommend Tironi's survey \cite{TPseudo}.

\section{On the tightness of a pseudoradial space}

\begin{theorem} \label{pseudothm}
Let $X$ be a Hausdorff pseudoradial space. Then $t(X) \leq L(X) \cdot F(X)$.
\end{theorem}

\begin{proof}
Let $\kappa=L(X) \cdot F(X)$ and suppose by contradiction that $t(X) > \kappa$. Then there is a $\kappa$-closed non-closed subset $A$ of $X$. Since $X$ is pseudoradial and $A$ is not closed we can find a transfinite sequence $S \subset A$  and a point $p \in \overline{A} \setminus A$ such that $S$ converges to $p$. 

For every $x \in A$, let $U_x$ and $V_x$ be disjoint open sets such that $x \in U_x$ and $p \in V_x$. Let $\mathcal{U}=\{U_x: x \in A\}$.

\medskip

\noindent {\bf Claim.} There is a family $\mathcal{W} \in [\mathcal{U}]^{\leq \kappa}$ such that $S \subset \bigcup \mathcal{W}$.

\begin{proof}[Proof of Claim] Assume this is not the case. Then we will construct by recursion a free sequence $F \subset S$ such that $|F| =\kappa^+$.

Suppose that for some $\beta < \kappa^+$ we have picked points $\{x_\alpha: \alpha < \beta \} \subset S$ and open families $\{\mathcal{U}_\alpha: \alpha < \beta \}$ such that:

\begin{enumerate}
\item $\overline{\{x_\alpha: \alpha < \tau \}} \subset \bigcup \mathcal{U}_\tau$, for every $\tau < \beta$.
\item $|\mathcal{U}_\alpha| \leq \kappa$, for every $\alpha < \beta$.
\renewcommand{\qedsymbol}{$\triangle$}
\end{enumerate}

The set $\overline{\{x_\alpha: \alpha < \beta\}}$ is contained in $A$ and has Lindel\"of degree at most $\kappa$, hence we can find a family $\mathcal{U}_\beta \subset \mathcal{U}$ such that $|\mathcal{U}| \leq \kappa$ and $\overline{\{x_\alpha: \alpha < \beta \}} \subset \bigcup \mathcal{U}_\beta$. Note that the family $\bigcup_{\alpha \leq \beta} \mathcal{U}_\alpha$ does not cover $S$ and hence we can choose a point $x_\beta \in S \setminus \bigcup_{\alpha \leq \beta} \mathcal{U}_\alpha$.

Eventually $\{x_\alpha: \alpha < \kappa^+\}$ is a free sequence of cardinality $\kappa^+$.
\renewcommand{\qedsymbol}{$\triangle$}
\end{proof}

Let $\lambda$ be the length of the sequence $S$. We can assume that $\lambda$ is a regular cardinal, and since $A$ is $\kappa$-closed we must have $\lambda \geq \kappa^+$. Therefore there must be $U \in \mathcal{W}$ such that $|S \cap U| =\lambda$. Now $U=U_x$, for some $x \in A$ and $V_x$ contains a final segment of $S$, because it is an open neighbourhood of $p$. But this contradicts the fact that $U_x$ and $V_x$ are disjoint.
\end{proof}

Bella \cite{B} obtained the following corollary with the additional assumption that $X$ is a regular space.

\begin{corollary}
Let $X$ be a Lindel\"of Hausdorff pseudoradial space. Then $t(X) = F(X)$.
\end{corollary}

Recalling that $\sigma_C(X)=t(X)$ for every almost radial space, we have the following pair of corollaries.

\begin{corollary}
Let $X$ be a Lindel\"of Hausdorff almost radial space. Then $\sigma_C(X) = F(X)$.
\end{corollary}

\begin{corollary} \label{almostchar}
Let $X$ be a Lindel\"of Hausdorff almost radial space. Then $X$ is sequential if and only if $F(X) \leq \omega$.
\end{corollary}

We can do without the pseudoradial assumption if we replace the tightness with a related cardinal invariant known as \emph{set-tightness}. The set-tightness of a topological space $X$ is defined as the minimum cardinal $\kappa$ such that for every non-closed subset $A$ of $X$ and for every point $p \in \overline{A} \setminus A$, there is a $\kappa$-sized family $\{A_\alpha: \alpha < \kappa \}$ of subsets of $A$ such that $p \in \overline{\bigcup \{A_\alpha: \alpha < \kappa \}}$, but $p \notin \bigcup \{\overline{A_\alpha}: \alpha < \kappa \}$. 

The set-tightness was introduced in \cite{AIT}, where it was called \emph{quasi-character}. Obviously $t_s(X) \leq  t(X)$. Arhangel'skii and Bella \cite{AB} proved that $t(X)=t_s(X)$ for every compact Hausdorff space $X$.

Bella \cite{B} proved that $t_s(X) \leq F(X)$ for every regular space $X$. We will prove that $t_s(X) \leq F(X) \cdot L(X)$ for every Hausdorff space $X$, so in particular, $t_s(X) \leq F(X)$ is true for every Lindel\"of Hausdorff space $X$. Note that there are Hausdorff pseudoradial spaces where $t_s(X)<t(X)$ (see \cite{ST}), so the bound $t(X) \leq F(X) \cdot L(X)$ for every pseudoradial Hausdorff space $X$ is not a consequence of $t_s(X) \leq F(X) \cdot L(X)$.

\begin{theorem}
Let $X$ be a Hausdorff space. Then $t_s(X) \leq F(X) \cdot L(X)$.
\end{theorem}

\begin{proof}
Let $\kappa=L(X) \cdot F(X)$ and suppose by contradiction that $t_s(X) > \kappa$. Then there are a set $A \subset X$ and a point $p \in \overline{A} \setminus A$ such that if $\{A_\alpha: \alpha < \kappa\}$ is a $\kappa$-sized family of subsets of $A$ such that $p \notin \bigcup \{\overline{A_\alpha}: \alpha < \kappa \}$ then $p \notin \overline{\bigcup \{A_\alpha: \alpha < \kappa \}}$.

For every $x \in X \setminus \{p\}$, let $U_x$ and $V_x$ be disjoint open sets such that $x \in U_x$ and $p \in V_x$. Let $\mathcal{U}=\{U_x: x \in X \setminus \{p\}\}$.

\medskip

\noindent {\bf Claim.} There is a family $\mathcal{W} \in [\mathcal{U}]^{\leq \kappa}$ such that $A\subset \bigcup \mathcal{W}$.

\begin{proof}[Proof of Claim] Assume this is not the case. Then we will construct a free sequence of cardinality $\kappa^+$ in $X$.

Suppose that for some $\beta < \kappa^+$ we have constructed points $\{x_\alpha: \alpha < \beta \} \subset A$ and open families $\{\mathcal{U}_\alpha: \alpha < \beta \}$ such that:

\begin{enumerate}
\item $\overline{\{x_\alpha: \alpha < \tau \}} \subset \bigcup \mathcal{U}_\tau$, for every $\tau < \beta$.
\item $|\mathcal{U}_\alpha| \leq \kappa$, for every $\alpha < \beta$.
\end{enumerate}

Note that $p \notin \overline{\{x_\alpha: \alpha < \beta \}}$, so $\mathcal{U}$ is a cover of $\overline{\{x_\alpha: \alpha < \beta \}}$. Moreover $L(\overline{\{x_\alpha: \alpha < \beta \}}) \leq \kappa$ and therefore we can find a subcollection $\mathcal{U}_\beta$ of $\mathcal{U}$ having cardinality $\leq \kappa$ such that $\overline{\{x_\alpha: \alpha < \beta \}} \subset \bigcup \mathcal{U}_\beta$. By our assumption $\bigcup \{\bigcup \mathcal{U}_\alpha: \alpha \leq \beta \}$ does not cover $A$ and hence we can pick a point $x_\beta \in A \setminus \bigcup \{\bigcup \mathcal{U}_\alpha: \alpha \leq \beta \}$

Eventually, $\{x_\alpha: \alpha < \kappa^+\}$ is a free sequence in $X$ having cardinality $\kappa^+$, which is a contradiction.
\renewcommand{\qedsymbol}{$\triangle$}
\end{proof}

Let $\{U_\alpha: \alpha < \kappa \}$ be an enumeration of $\mathcal{W}$ and set $A_\alpha=U_\alpha \cap A$, for every $\alpha < \kappa$. Then $\{A_\alpha: \alpha < \kappa \}$ is a $\kappa$-sized family of subsets of $A$ such that $p \notin \bigcup \{ \overline{A_\alpha}: \alpha < \kappa \}$ but $p \in \overline{\bigcup \{A_\alpha: \alpha < \kappa \}}$, and that is a contradiction.
\end{proof}

\begin{corollary}
Let $X$ be a Lindel\"of Hausdorff space. Then $t_s(X)=F(X)$.
\end{corollary}

\section{The $G_\delta$ topology}

In \cite{DJSSW}, Dow, Juh\'asz, Soukup, Szentmikl\'ossy and Weiss proved the following bound for the tightness of the $G_\delta$ topology of a Lindel\"of regular space.

\begin{theorem} \cite{DJSSW}
Let $X$ be a Lindel\"of regular space. Then $t(X_\delta) \leq 2^{t(X)}$.
\end{theorem}

In the next theorem we will show how to relax the regularity assumption in their bound.

\begin{theorem} \label{tightthm}
Let $(X, \tau)$ be a Lindel\"of Hausdorff space. Then $t(X_\delta) \leq 2^{t(X)}$.
\end{theorem}

\begin{proof}
Let $A \subset X_\delta$ be a non-closed set and $p$ be a point in the $G_\delta$-closure of $A$. Let $\kappa=t(X)$, let $\theta$ be a large enough regular cardinal and let $M$ be a $\kappa$-closed elementary submodel of $H(\theta)$ such that $X, \tau, A, p \in M$ and $|M|=2^\kappa$. Note that if we call $\tau_\delta$ the collection of all $G_\delta$ subsets of $X$ we also have $\tau_\delta \in M$. We claim that $p \in Cl_\delta(A \cap M)$.

Indeed, let $G$ be a $G_\delta$ set such that $p \in G$ and let $\{U_n: n < \omega\}$ be a sequence of open sets such that $G = \bigcap \{U_n: n < \omega \}$.

For every point $x \in \overline{A \cap M}$, such that $x \neq p$, let $V_x$ and $W_x$ be disjoint open sets such that $x \in V_x$ and $p \in W_x$.  Since $t(X) \leq \kappa$, we can find, for every $x \in \overline{A \cap M}$ a $\leq \kappa$-sized subset $S_x \subset A \cap M$ such that $x \in \overline{S_x}$. Since $M$ is $\kappa$-closed, we have $S_x \in M$ and hence $\overline{S_x} \in M$. Now $\overline{S_x}$ is Lindel\"of and $\{V_x: x \in \overline{A \cap M} \setminus \{p\} \}$ is an open cover of $\overline{S_x}$. Therefore we can find a countable subset $C_x \subset \overline{A \cap M} \setminus \{p\}$ such that $\overline{S_x} \subset \bigcup \{V_x : x \in C_x \}$ and $p \in \bigcap \{W_x: x \in C_x \}$. Define $O'_x=\bigcup \{V_x: x \in C_x \}$ and $G'_x=\bigcap \{W_x: x \in C_x \}$. Note that $O'_x$ is an open set, $G'_x$ is a $G_\delta$ set and $O'_x \cap G'_x=\emptyset$. So we just showed that:

$$H(\theta) \models (\exists O) (\exists G)(O \in \tau \wedge G \in \tau_\delta \wedge \overline{S_x} \subset O \wedge p \in G \wedge O \cap G=\emptyset)$$

Note that all free variables in the above formula (namely $\overline{S_x}, \tau$ and $\tau_\delta$) are in $M$ and hence by elementarity we also have:

$$M \models (\exists O) (\exists G)(O \in \tau \wedge G \in \tau_\delta \wedge \overline{S_x} \subset O \wedge p \in G \wedge O \cap G=\emptyset)$$

What this means is that we can find, for every $x \in \overline{A \cap M} \setminus \{p\}$ an open set $O_x \in M$ and a $G_\delta$ set $G_x \in M$ such that $\overline{S_x} \subset O_x$, $p \in G_x$ and $O_x \cap G_x=\emptyset$.

Note now that $\{O_x: x \in \overline{A \cap M} \setminus \{p\} \}$ is an open cover of the Lindel\"of space $\overline{A \cap M} \setminus U_n$, for every $n < \omega$. Therefore we can find a countable subset $C_n$ of $\overline{A \cap M} \setminus U_n$ such that $\overline{A \cap M} \setminus U_n \subset \bigcup \{O_x: x \in C_n \}$. Note that $H_n=\bigcap \{G_x: x \in C_n\}$ is a $G_\delta$ set and $H_n \cap \overline{A \cap M} \setminus U_n=\emptyset$. Let now $H=\bigcap \{H_n: n < \omega \} \in M$. Note that $H$ is a $G_\delta$ set containing $p$ and therefore $H \cap A \neq \emptyset$. Since $H, A \in M$, by elementarity we also have $H \cap A \cap M \neq \emptyset$. But $H \cap \overline{A \cap M} \setminus G=\emptyset$, that is $H \cap \overline{A \cap M} \subset G$. It follows that $G \cap A \cap M \neq \emptyset$, as we wanted.

\end{proof}

\begin{lemma} \label{lindgdelta}
(Carlson, Porter and Ridderbos \cite{CPR}) Let $X$ be a Lindel\"of Hausdorff space. Then $L(X_\delta) \leq 2^{F(X)}$,
\end{lemma}

\begin{corollary}
Let $X$ be a Lindel\"of Hausdorff pseudoradial space. Then $F(X_\delta) \leq 2^{F(X)}$.
\end{corollary}

\begin{proof}
Combining Theorem $\ref{pseudothm}$ with Theorem $\ref{tightthm}$ we obtain $t(X_\delta) \leq 2^{t(X)} = 2^{F(X)}$. Recalling that $F(X_\delta) \leq L(X_\delta) \cdot t(X_\delta)$ and using Lemma $\ref{lindgdelta}$ we obtain $F(X_\delta) \leq 2^{F(X)}$, as desired.
\end{proof}

The above corollary was proved in \cite{BS2} with the additional assumption that $X$ is a regular space.

\section{Acknowledgements}

The author is grateful to the referees for their careful reading of the paper and to INdAM-GNSAGA for partial financial support.

\end{document}